\documentclass[11pt]{amsart}
\usepackage[margin=1in]{geometry}  
\usepackage{amssymb}
\usepackage[all]{xy}

\newcommand{\pd}{\operatorname{pd}}
\newcommand{\Ker}{\operatorname{Ker}}
\newcommand{\im}{\operatorname{Im}}

\newcommand{\rank}{\operatorname{rank}}

\def\a{{\mathbf a}}
\def\aa{{\alpha}}

\def\N{{\mathbb N}}
\def\B{{\mathcal B}}

\def\Z{{\mathbb Z}}

\def\m{{\mathbf m}}
\def\e{{\mathbf e}}

\def\b{{\mathbf b}}
\def\cc{{\mathbf c}}
\def\d{{\mathbf d}}
\def\u{{\mathbf u}}
\def\v{{\mathbf v}}
\def\t{{\mathbf t}}
\def\x{{\mathbf x}}
\def\y{{\mathbf y}}
\def\z{{\mathbf z}}
\def\w{{\mathbf w}}
\def\iBetti{{$i$-Betti~~}} 

\newtheorem{theorem}{Theorem}[section]
\newtheorem{lemma}[theorem]{Lemma}

\theoremstyle{definition}

\newtheorem{corollary}[theorem]{Corollary}
\newtheorem{proposition}[theorem]{Proposition}
\newtheorem{example}[theorem]{Example}

\theoremstyle{remark}
\newtheorem{remark}[theorem]{Remark}

\numberwithin{equation}{section}

\begin{document}

\title{Gluing semigroups and strongly indispensable free resolutions}

\author{Mesut \c{S}ah\. in}
\address{Department of Mathematics,
Hacettepe University, Beytepe,  06800, Ankara, Turkey}
\email{mesut.sahin@hacettepe.edu.tr}

\author{Leah Gold Stella}
\address{Department of Mathematics,
Cleveland State University, Cleveland, OH, USA}
\email{l.gold33@csuohio.edu}
\thanks{The first author is supported by the project 114F094 under the program 1001 of the
Scientific and Technological Research Council of Turkey.}
\date{\today}


\commby{}

\dedicatory{}

\begin{abstract}
We study strong indispensability of minimal free resolutions of semigroup rings focusing on the operation of gluing used in literature 
to take examples with a special property and produce new ones.
We give a naive condition to determine whether gluing of two semigroup rings has a strongly indispensable minimal free resolution. As applications, we determine simple gluings of $3$-generated non-symmetric, $4$-generated symmetric and pseudo symmetric numerical semigroups as well as obtain infinitely many new complete intersection semigroups of any embedding dimensions, having strongly indispensable minimal free resolutions.
\end{abstract}

\maketitle

\section{introduction} 
Let $\N$ denote the set of non-negative integers and consider the affine semigroup $S$ generated minimally by $\m_1,\dots,\m_n \in \N^r$. 
Let $K$ be a field. 
Turning the additive structure of $S$ into a multiplicative one yields an algebra $K[S]$
called the affine semigroup ring associated to $S$. Any polynomial ring $R=K[x_1,\dots,x_n]$, can be graded by $S$, via $\deg _S (x_i)=\m_i$, yielding a graded map $R \rightarrow K[S]$, sending $x_i$ to $\t^{\m_i}:=t_1^{m_{i1}}\cdots t_r^{m_{ir}}$, whose kernel, denoted by $I_S$, is called the toric ideal of $S$. When $K$ is algebraically closed, $K[S]$ is isomorphic to the coordinate ring $R/I_S$ of the affine toric variety $V(I_S)$.

 Toric ideals with unique minimal generating sets or equivalently those that are generated by indispensable binomials attracted researchers attention 
 due to its importance for algebraic statistics. This 
 connection
 leads to a search for criteria to characterize indispensability (see e.g. \cite{ CKT,CTV,go,koj,ov,tak}). Indispensable binomials are those that appear in every minimal \textit{binomial} generating set up to a constant multiple. Strongly indispensable binomials are those appearing in \textit{every} minimal generating set, up to a constant multiple. In the same vein, as introduced for the first time by Charalambous and Thoma in \cite{haraCM,haraJA}, strongly indispensable higher syzygies are those appearing in \textit{every} minimal free resolution. Semigroups all of whose higher syzygy modules are generated minimally by strongly indispensable elements are said to have a strongly indispensable minimal free resolution, SIFRE for short. The statistical models having SIFREs or equivalently having uniquely generated higher syzygy modules are a subclass of those having a unique Markov basis and therefore have a better potential statistical behaviour. 

It is difficult to construct examples having SIFREs.  
It is known that generic 
lattice ideals have SIFRE (\cite[Theorem 4.2]{peeva}, \cite[Theorem 4.9]{haraCM}). Numerical semigroups having SIFREs have been classified for some small embedding dimensions in \cite{bafrsa,pseudo}. 

Motivated by the third question stated by Charalambous and Thoma at the end of \cite{haraJA}, our main aim in this article is to identify some semigroups having SIFREs. We focus on the operation of gluing used in literature to produce more examples with a special property from the existing one (see e.g. \cite{watanabe,morales,ext,numata,fel}). In Section $2$, we restate the general method given by \cite[Theorem 4.9]{haraCM} to check if a given semigroup has a SIFRE, see Lemma \ref{indispensable}. In section $3$, we study the gluing $S$ of $S_1$ and $S_2$. We show that a minimal graded free resolution for $K[S]$ is obtained from that of $K[S_1]$ and $K[S_2]$ via the tensor product of three complexes (for details see Theorem \ref{res}). As a consequence we get the Betti $S$-degrees, see Lemma \ref{betti}, which is key for our refined criterion special to semigroups obtained by gluing. We then give a naive criterion to determine whether $K[S]$ has a SIFRE, see Theorem \ref{gluing}. We conclude the section with Example \ref{eg} illustrating the efficiency of our criterion. In the last section, we focus on a particular gluing also known as extension or simple gluing, and get an even more refined criterion 
in this case.
 It turns out that this condition is very helpful for producing infinitely many examples having SIFRE from a single example. As applications, we determine extensions of $3$-generated non-symmetric, $4$-generated symmetric and pseudo symmetric numerical semigroups as well as obtain infinitely many complete intersection semigroups of any embedding dimension, having SIFREs.

\section{strongly indispensable minimal free resolutions}

Let $({\bf F},\phi)$ be a graded minimal free $R$-resolution of $K[S]$, where
$${\bf F}: \ 0\longrightarrow R^{\beta_{k}}\stackrel{\phi_{k}}{\longrightarrow}R^{\beta_{k-1}}\stackrel{\phi_{k-1}}{\longrightarrow}\cdots\stackrel{\phi_2}{\longrightarrow}R^{\beta_1}\stackrel{\phi_1}{\longrightarrow}R^{\beta_0}{\longrightarrow}K[S]{\longrightarrow}0. $$ 

The elements $s_{i,j}\in S$ for which $\displaystyle R^{\beta_{i}}=\bigoplus_{j=1}^{\beta_i} R[{-s_{i,j}}]$ are called 
\iBetti
$S$-degrees. Denote by $\B_i(S)$ the set of these \iBetti $S$-degrees for $1\leq i \leq \pd(S)$ and let $\B_i(S)=\{0\}$ otherwise, where $\pd(S)$ is the projective dimension of $K[S]$. Note that we allow $\B_i(S)$ to contain repeating elements in a nonstandard way for convenience.

The resolution $({\bf F},\phi)$ is called \textbf{strongly indispensable} if for any graded minimal resolution $({\bf G},\theta)$, we have an injective complex map
$i\colon({\bf F},\phi)\longrightarrow({\bf G},\theta)$. When $({\bf F},\phi)$ is strongly indispensable $S$ or $K[S]$ is said to have a SIFRE for short.

The following general criterion about strong indispensability is a version of Charalambous and Thoma's Theorem~4.9 in \cite{haraCM} stated slightly different for semigroup rings. We compare two elements $s_1$ and $s_2$ of $S$ saying that $s_1 < s_2$ if $s_2-s_1 \in S$. An element is regarded minimal with respect to this partial ordering.
\begin{lemma} \label{indispensable} A minimal graded free resolution of $K[S]$ is strongly indispensable if and only if $\pm (b_{i}-b^{'}_{i}) \notin S$ for all $b_{i},b^{'}_{i}\in \B_i(S)$ and for each $1\leq i \leq \pd(S)$. 
\end{lemma}
\begin{proof} It follows from Theorem $4.9$ in \cite{haraCM} that $K[S]$ has a SIFRE if and only if \iBetti degrees are minimal elements of $\B_i(S)$ and are different, for each $i$. If \iBetti degrees are different and minimal, for each $i$, then their differences can not lie in $S$ as otherwise there would be $b_{i},b^{'}_{i}\in \B_i(S)$ with $b_{i}-b^{'}_{i}=s \in S\setminus \{0\}$, contradicting the minimality of $b_{i}$. Conversely, if 
$\pm (b_{i}-b^{'}_{i}) \notin S$ for all $b_{i},b^{'}_{i}\in \B_i(S)$ and for each $1\leq i \leq \pd(S)$, then all $b_{i}\in \B_i(S)$ are clearly minimal. They are also different as $S$ always contains $0$.
\end{proof}
When  $S$ is symmetric, it is sufficient to check  the condition above for  the first half of the indices. 
\begin{lemma} \label{symmetriclemma} 
If $S$ is symmetric, then $K[S]$ has a SIFRE if and only if $\pm (b_{i}-b^{'}_{i}) \notin S$ for all $b_{i},b^{'}_{i}\in \B_i(S)$ and for each $1\leq i \leq \lfloor \pd(S)/2 \rfloor$.
\end{lemma}
\begin{proof} The proof of Lemma~21 in Barucci, Fr\"oberg, and \c{S}ahin's paper \cite[Lemma 21]{bafrsa} extends from numerical semigroups to arbitrary affine semigroups, as it uses the symmetry in the minimal graded free resolution of $K[S]$, which is true for any graded Gorenstein $K$-algebra by Stanley's second proof of Theorem $4.1.$ in \cite{sta}. 
\end{proof}

We finish this section by illustrating how this criterion applies.
\begin{example}\label{eg} Let $S=\langle 5\cdot31,5\cdot37,5\cdot41,82\cdot4,82\cdot5\rangle=\langle 155, 185, 205, 328, 410\rangle$. Macaulay 2 computes $I_S$ to be the following ideal

\begin{verbatim}
       2        5    2       9    2   2   5    4   7 3    4
I = ( x  - x , x  - x x x , x  - x x x , x  - x , x x  - x  ).
       3    5   2    1 3 5   1    2 3 5   4    5   1 2    5
\end{verbatim}
In order to check whether $S$ has a SIFRE we compute a minimal $S$-graded free resolution using the commands:
\begin{verbatim}
C = res I;
C.dd   
\end{verbatim}
and determine all the $i$-Betti $S$-degrees as follows. For instance, the following computes the set $\B_1(S)$ of $1$-Betti $S$-degrees:
\begin{verbatim}
i1 : B1=(degrees C.dd#1)_1
o1 = {410, 925, 1395, 1640, 1640}
\end{verbatim}

As there are two syzygies with the same Betti $S$-degree $1640$, and their diffference is $0\in S$, the semigroup $S$ can not have a SIFRE. 
 \end{example} 

\section{Gluing Strongly Indispensable Resolutions}
In this section, we study the concept of gluing introduced for the first time by Rosales \cite{ros}. 
 Let $S_1=\N\{\mathbf a_1,\dots,\mathbf a_m\}$ and $S_2=\N\{\mathbf b_1,\dots,\mathbf b_n\}$ be two affine semigroups. If there is an $\aa \in S_1 \cap S_2$ such that $\Z S_1 \cap \Z S_2=\Z \aa$ then $S=S_1+S_2$ is said to be the \textbf{gluing} of $S_1$ and $S_2$ by the virtue of \cite[Theorem 1.4 and Definition 2.1]{ros}. When $$  \aa=u_1\mathbf a_1+\cdots+u_m \mathbf a_m=v_1\mathbf b_1+\cdots+v_n \mathbf b_n,$$ the binomial $f_{\aa}=x_1^{u_1}\cdots x_m^{u_m}-y_1^{v_1}\cdots y_n^{v_n}$ has $S$-degree $\aa$ and the toric ideal is of the form $$I_S=I_{S_1}+I_{S_2}+\langle f_{\aa} \rangle \subset R=K[x_1,\dots, x_m, y_1, \dots, y_n].$$ Note that $f_{\aa}$ might not be unique as different $u_i$'s or $v_j$'s may appear in the expression of $\aa$ above. Let
$${\bf F}: \ 0\rightarrow F_k \stackrel{\phi_{k}}{\rightarrow}\cdots\stackrel{\phi_2}{\rightarrow}F_1\stackrel{\phi_1}{\rightarrow}F_0{\rightarrow}0 $$ be a 
minimal $S_1$-graded free resolution of $I_{S_1}$ with $H_0({\bf F})=R/I_{S_1}$,  $${\bf G}: \ 0\rightarrow G_l \stackrel{\Phi_{l}}{\rightarrow}\cdots\stackrel{\Phi_2}{\rightarrow}G_1\stackrel{\Phi_1}{\rightarrow}G_0{\rightarrow}0 $$  be a 
minimal $S_2$-graded free resolution of $I_{S_2}$ with $H_0({\bf G})=R/I_{S_2}$. 

Our aim is to compute a minimal $S$-graded free resolution of $I_{S}$ using the complexes $\bf F$ and $\bf G$. Since $I_S=I_{S_1}+I_{S_2}+\langle f_{\aa} \rangle$ the idea is to tensor these complexes and the complex below :
$${\bf C_{f_{\aa}}}:  0\rightarrow R \stackrel{f_{\aa}}{\longrightarrow}R \rightarrow 0 .$$ This method works if $f_{\aa}$ is a non-zero-divisor on $R/(I_{S_1}+I_{S_2})$ so we address it first. 

\begin{lemma} The gluing binomial $f_{\aa}$ is a non zerodivisor on $R/(I_{S_1}+I_{S_2})$. 
 \end{lemma}
 \begin{proof}
 For notational convenience, let $f_{\aa}=\x^{\u}-\y^{\v}=x_1^{u_1}\cdots x_m^{u_m}-y_1^{v_1}\cdots y_n^{v_n}$. Take an element $g=\sum_{\z,\w}c_{\z,\w}\x^{\z}\y^{\w}\in R$ with $gf_{\aa}\in I_{S_1}+I_{S_2}$. As $I_{S_1}+I_{S_2}$ is generated by binomials of the form $\x^{\z}-\x^{\z'}$ and $\y^{\w}-\y^{\w'}$, these binomials appear in the expansion of $gf_{\aa}=\sum_{\z,\w}c_{\z,\w}(\x^{\z+\u}\y^{\w}-\x^{\z}\y^{\w+\v})$. In other words, each monomial $\x^{\z+\u}\y^{\w}$ has a match of type $\x^{\z'+\u}\y^{\w'}$ or $\x^{\z'}\y^{\w'+\v}$ such that 
 $$\x^{\z+\u}\y^{\w}-\x^{\z'+\u}\y^{\w'} \quad \mbox{or}\quad \x^{\z+\u}\y^{\w}-\x^{\z'}\y^{\w'+\v}$$
  is divisible by one of the binomials $\x^{\z}-\x^{\z'}$ or $\y^{\w}-\y^{\w'}$. In the first case, this is possible only if $\z=\z'$ or  $\w=\w'$. If $\z=\z'$, then 
 $$\x^{\z+\u}\y^{\w}-\x^{\z'+\u}\y^{\w'}=\x^{\z+\u}(\y^{\w}-\y^{\w'})=\x^{\u}(\x^{\z}\y^{\w}-\x^{\z}\y^{\w'}).$$ This means that the term $\x^{\z}\y^{\w}$ of $g$ has a match $\x^{\z}\y^{\w'}$ such that $\x^{\z}\y^{\w}-\x^{\z}\y^{\w'}$ is divisible by $\y^{\w}-\y^{\w'}$. Similarly, one can prove that this happens for the other cases. Hence, terms in $g$ may be rearranged so that it is an algebraic combination of binomials $\x^{\z}-\x^{\z'}$ and $\y^{\w}-\y^{\w'}$, that is, $g\in I_{S_1}+I_{S_2}$.
 \end{proof}
 
We are now ready to prove the following key result.
 
 \begin{theorem}\label{res} Let $S$ be the gluing of $S_1$ and $S_2$. If ${\bf F}$ is a 
 minimal $S_1$-graded free resolution of $I_{S_1}$ and ${\bf G}$ is a minimal $S_2$-graded free resolution of $I_{S_2}$, then ${\bf   C_{f_{\aa}} \otimes F \otimes G}$ is aminimal $S$-graded free resolution of $I_S$.
 \end{theorem}
 
 \begin{proof}Recall that the tensor product of $\bf F$ and $\bf G$ is a complex
$${\bf F \otimes G}: \ 0\longrightarrow  F_k \otimes G_l \stackrel{\delta_{k+l}}{\longrightarrow}\cdots\stackrel{\delta_2}{\longrightarrow} F_1\otimes G_0 \oplus F_0\otimes G_1 \stackrel{\delta_1}{\longrightarrow}F_0\otimes G_0 {\longrightarrow}0 $$
with terms $ ({\bf F \otimes G})_i= \oplus_{p+q=i}^{} F_p \otimes G_q$ and maps given by $$\delta_i(\sum_{p+q=i} a_p\otimes b_q)=\sum_{p+q=i}\phi_p(a_p)\otimes b_q+(-1)^p a_p\otimes \Phi_q (b_q).$$
It is well known that $H_i({\bf F \otimes G})=H_i({\bf F \otimes R/I_{S_2}})$ where ${\bf F \otimes R/I_{S_2}}$ is the complex $$\ 0\rightarrow F_k\otimes R/I_{S_2} \stackrel{\Delta_{k}}{\rightarrow}\cdots\stackrel{\Delta_2}{\rightarrow}F_1\otimes R/I_{S_2} \stackrel{\Delta_1}{\rightarrow}F_0\otimes R/I_{S_2} {\rightarrow}0 ,$$ with $\Delta_i(a_i\otimes b)=\phi_i(a_i)\otimes b$. It is easy to see that $$\displaystyle \Ker(\Delta_i)=[\Ker(\phi_i)\otimes R/I_{S_2}] \cup [(\phi_i^{-1}(\bigoplus_{j=1}^{r_{i-1}}I_{S_2})\otimes R/I_{S_2}],$$
where $r_{i-1}=\rank(F_{i-1})$, and $\im (\Delta_i)=\im(\phi_i)\otimes R/I_{S_2}$. Since, $\im(\phi_i)$ involves the variables $x_j$ only and $I_{S_2}$ involves the variables $y_j$ only, it follows that $\phi_i^{-1}(\oplus_{j=1}^{r_{i-1}}I_{S_2})=\{0\}.$ Thus, $H_i({\bf F \otimes G})=H_i({\bf F \otimes R/I_{S_2}})=0$, for all $i>0$. Since $H_0({\bf F \otimes G})= R/I_{S_1} \otimes R/I_{S_2} \cong R/(I_{S_1}+I_{S_2})$, it follows that ${\bf F \otimes G}$ is an $S$-graded minimal free resolution of $R/(I_{S_1}+I_{S_2})$.

Now let $f=f_{\aa}$ for notational convenience. As before, 
$$H_i({\bf   C_f \otimes F \otimes G})=H_i({\bf C_f \otimes R/(I_{S_1}+I_{S_2})}), \quad \mbox{where} \quad {\bf C_f \otimes R/(I_{S_1}+I_{S_2})}\quad \mbox{is}$$  
 $$0\rightarrow R \otimes R/(I_{S_1}+I_{S_2}) \stackrel{ f \otimes 1}{\longrightarrow}R \otimes R/(I_{S_1}+I_{S_2})\rightarrow 0.$$
 Note that $H_1({\bf C_f \otimes R/(I_{S_1}+I_{S_2})})=((I_{S_1}+I_{S_2}):f)\otimes R/(I_{S_1}+I_{S_2})=\{0\}$ as $f$ is a non-zero-divisor on $R/(I_{S_1}+I_{S_2})$. Since we have the following isomorphism $$H_0({\bf C_f \otimes R/(I_{S_1}+I_{S_2})})\cong R/(I_{S_1}+I_{S_2}+\langle f \rangle ),$$ it follows that ${\bf   C_f \otimes F \otimes G}$ gives an $S$-graded minimal free resolution of $I_S$.
 \end{proof}
\begin{remark} As we were preparing the final version for submission, a slightly different version of the theorem above is posted on arxiv by Gimenez and Srinivasan \cite{gs}. See also our preprint posted on arxiv at https://arxiv.org/abs/1710.09298. 
\end{remark}
 
Recall that $\B_i(S)$ is the set of \iBetti $S$-degrees of a minimal free resolution of $K[S]$ for every $1\leq i \leq \pd(S)$ and $\B_i(S)=\{0\}$ otherwise. 
 
\begin{lemma} \label{betti} Let $S$ be the gluing of $S_1$ and $S_2$. Then,
$$\B_{i}(S)=\left[ \bigcup_{p+q=i}^{} \B_{p}(S_1)+\B_{q}(S_2)\right]\cup \left[ \bigcup_{p+q=i-1}^{} \B_{p}(S_1)+\B_{q}(S_2)+\{\aa\} \right].$$

\end{lemma}
\begin{proof} By Theorem \ref{res}, ${\bf   C_{f_{\aa}} \otimes F \otimes G}$ is an $S$-graded minimal free resolution of $I_S$. Hence, the proof follows from the following
$$(C_{f_{\aa}} \otimes F \otimes G)_i=\bigoplus_{p+q=i}^{} R \otimes F_p \otimes G_q+\bigoplus_{p+q=i-1}R(-\aa) \otimes F_p \otimes G_q, $$
since $S$-degrees of elements in $F_p \otimes G_q$ constitute the set $\B_{p}(S_1)+\B_{q}(S_2)$.
\end{proof}

We use the following simple observation in the proof of our main result.

\begin{lemma} \label{dif} Let $S$ be the gluing of $S_1$ and $S_2$. Fix  $j \in\{1,2\}$, and $b,b' \in S_j$. Then, $b-b'\in S_j \iff b-b'\in S$.
\end{lemma}
\begin{proof} Without loss of generality, assume that $j=1$.
As $S_1 \subset S$, $b-b'\in S_1 \Rightarrow b-b'\in S$. For the converse, take $b,b' \in S_1$ with $b-b'\in S$. Then, $b-b'=s_1+s_2$, for some $s_1\in S_1$ and $s_2\in S_2$, since $S=S_1+S_2$. So, $s_2=b-b'-s_1\in \Z S_1$. Since $\Z S_1 \cap \Z S_2=\Z \aa$ and $s_2\in S_2$, we have $s_2=k\aa$ for a positive integer $k$. Hence,  $b-b'=s_1+k\aa \in S_1$.
\end{proof}
 We are now ready to prove our main result which gives a practical method to produce infinitely many affine semigroups having a SIFRE.
\begin{theorem}\label{gluing} Let $b_{i,j}$ denote an element of $ \B_{i}(S_j)$ for $i=1,\dots, pd (S_j)$, $j=1,2$. Then, $I_S$ has a SIFRE if and only if $I_{S_1}$ and $I_{S_2}$ have SIFREs and the following hold
\begin{enumerate} 
\item $\pm (\aa+b_{i-1,j}-b_{i,j}) \notin S_j$, \label{cond1}
\item $\pm (b_{p,1}+b_{q,2}-b'_{r,1}-b'_{s,2}) \notin S$, for $p-r \geq 2$, where $p+q=i=r+s$, \label{cond2}
\item $\pm (b_{p,1}+b_{q,2}-b'_{r,1}-b'_{s,2}-\aa) \notin S$ for $p-r \geq 2$, where $p+q=i=r+s+1$. \label{cond3}
\end{enumerate}

\end{theorem} 
\begin{proof} Let us prove necessity first. 
By Lemma \ref{indispensable}, the differences between the elements in $\B_i(S)$ do not belong to $S$. Let $b_{i,j},b'_{i,j}\in \B_i(S_j)$, for $j=1,2$. By Lemma \ref{betti}, we have $\B_i(S_j) \subset \B_i(S)$, and thus $b_{i,j}-b'_{i,j}\notin S$. This implies $b_{i,j}-b'_{i,j}\notin S_j$ by Lemma \ref{dif}, which means that $I_{S_1}$ and $I_{S_2}$ have SIFRE by the virtue of Lemma \ref{indispensable}. As the elements in the Conditions (\ref{cond1})-(\ref{cond3}) are the differences of some elements in $\B_i(S)$, they do not belong to $S$. So, Conditions~(\ref{cond2}) and (\ref{cond3}) hold. Lemma \ref{dif} implies (\ref{cond1}) now.

Now let us prove sufficiency.  If $b,b' \in \B_{i}(S)$, then there are three possibilities due to Lemma \ref{betti}:\\ (\textbf{i}) $b,b' \in \B_{p}(S_1)+\B_{q}(S_2)$, for $p+q=i$,\\
(\textbf{ii}) $b,b' \in \B_{p}(S_1)+\B_{q}(S_2)+\aa$, for $p+q=i-1$,\\
(\textbf{iii}) $b\in \B_{p}(S_1)+\B_{q}(S_2)$, $b' \in \B_{r}(S_1)+\B_{s}(S_2)+\aa$, for $p+q=i=r+s+1$.

Case (\textbf{i}): Let $ b=b_{p,1}+b_{q,2}$ and $b'=b'_{r,1}+b'_{s,2}$ with $p+q=i=r+s$. Suppose now that $b-b'\in S$. Then,   $b-b'=b_{p,1}+b_{q,2}-b'_{r,1}-b'_{s,2}=s_1+s_2$, for some $s_1\in S_1$ and $s_2\in S_2$. Thus, $b_{p,1}-b'_{r,1}-s_1=s_2-b_{q,2}+b'_{s,2}=k\aa$ being an element of $ \Z S_1 \cap \Z S_2=\Z \aa$. By Condition (\ref{cond2}), we need only to check the difference for $p=r$ and $p=r+1$.

When $p=r$, it follows that $b_{p,1}-b'_{p,1}=s_1+k\aa\in S_1$ if $k \geq 0$, and that $b_{q,2}-b'_{q,2}=s_2+(-k)\aa\in S_2$ if $k<0$, contradicting to hypothesis by Lemma \ref{indispensable}. 

When $p=r+1$, it follows that $b_{r+1,1}-b'_{r,1}-s_1=s_2-b_{q,2}+b'_{q+1,2}=k\aa$. Since the resolution of $I_{S_2}$ is $S_2$-graded, there is $s'_2\in S_2$ such that $b'_{q+1,2}=b_{q,2}+s'_2$. So, $k\aa=s_2+s'_2\in S_2$. Since $S_2 \cap (-S_2)=\{0\}$, we have $k>0$. But then,
$b_{r+1,1}-b'_{r,1}-\aa=s_1+(k-1)\aa \in S_1$, which contradicts Condition $(\ref{cond1})$.

Case (\textbf{ii}): follows from Case (\textbf{i}).

Case (\textbf{iii}): Let $ b=b_{p,1}+b_{q,2}$ and $b'=b'_{r,1}+b'_{s,2}+\aa$ with $p+q=i=r+s+1$. Suppose that $b-b'\in S$. Then, $b-b'=b_{p,1}+b_{q,2}-b'_{r,1}-b'_{s,2}-\aa=s_1+s_2$, for some $s_1\in S_1$ and $s_2\in S_2$. It follows that $b_{p,1}-b'_{r,1}-s_1=s_2-b_{q,2}+b'_{s,2}+\aa=k\aa$, for some $k\in \Z$. By Condition (\ref{cond3}), we need only to check the difference for $p=r$ and $p=r+1$.

When $p=r$, we have $b_{r,1}-b'_{r,1}=s_1+k\aa\in S_1$ if $k > 0$, and $b_{s+1,2}-b'_{s,2}-\aa=s_2+(-k)\aa\in S_2$ if $k\leq0$, which give rise to a contradiction.

When $p=r+1$, we have $b_{r+1,1}-b'_{r,1}-\aa=s_1+(k-1)\aa\in S_1$ if $k > 0$, and $b_{s,2}-b'_{s,2}=s_2+(-k)\aa\in S_2$ if $k\leq0$, which give rise to a contradiction.

One proves that $b'-b\notin S$ similarly.
\end{proof}

\begin{remark} \label{numericalgluing} Let $S_1$ and
$S_2$ be two numerical semigroups minimally generated by the integers
$a_{1}<\dots<a_{m}$ and $b_{1}<\dots<b_{n}$ respectively. This implies that $\gcd(a_{1},\dots,a_{m})=\gcd(b_{1},\dots,b_{n})=1$. Take $a=u_1a_1+\dots+u_ma_m \in S_1$ and $b=v_1b_1+\dots+v_nb_n \in S_2$. Then, by \cite[Lemma 2.2]{ros}, the
numerical semigroup $S=\langle
ba_{1},\dots,ba_{m},ab_{1},\dots,ab_{n} \rangle$ is a gluing
of the semigroups $bS_1$ and $aS_2$ if and only if $gcd(a,b)=1$ with $a \not\in
\{a_{1},\dots,a_{m}\}$ and $b \not\in \{b_{1},\dots,b_{n}\}$ such that
$$\{ba_{1},\dots,ba_{m}\} \cap \{ab_{1},\dots,ab_{n}\}=\emptyset.$$
In this case, one needs to pay attention to the notation as $S$ is not a gluing of $S_1$ and $S_2$. For instance, one has to use $\B_p(bS_1)$ in Lemma \ref{betti} rather than $\B_p(S_1)$.
\end{remark}
The following illustrates the efficiency of our criterion special to semigroups obtained by gluing.
\begin{example}\label{eg}   $S_1=\langle 31,37,41\rangle$ and $S_2=\langle 4,5\rangle$ have SIFREs. Take $$a=u_1\cdot 31+u_2\cdot 37+u_3\cdot 41 \in S_1 \quad \mbox{and} \quad b=v_1\cdot4+v_2\cdot5 \in S_2.$$ Then, by Remark \ref{numericalgluing}, the semigroup $S=bS_1+aS_2$ is the gluing of $bS_1$ and $aS_2$ if and only if $gcd(a,b)=1$ with $a \not\in
\{31,37,41\}$ and $b \not\in \{4,5\}$. It is difficult, however, to determine $a$ and $b$ for which $S$ has a SIFRE using Lemma \ref{indispensable} as we explain now. Macaulay 2 computes the $i$-Betti $S_1$-degrees as follows:
\begin{verbatim}
i1 : B1S1=(degrees C1.dd#1)_1
o1 = {185,279,328}
i2 : B2S1=(degrees C1.dd#2)_1
o2 = {390,402}
\end{verbatim}
Clearly, the only Betti $S_2$-degree is $\B_1(S_2)=\{20\}$.
Therefore, using Lemma \ref{betti} we get the following sets:
\begin{itemize}
\item $\B_1(S)=\{185b,\quad 279b,\quad 328b,\quad 20a,\quad ab\},$ 
 \item $\B_2(S)=\{390b,\quad 402b,\quad 185b+20a,\quad 279b+20a,\quad 328b+20a,\\
 185b+ab,\quad 279b+ab,\quad 328b+ab,\quad 20a+ab\},$ 
 \item $\B_3(S)=\{390b+20a,\quad 402b+20a,\quad 390b+ab,\quad 402b+ab,\\
 185b+20a+ab,\quad 279b+20a+ab,\quad 328b+20a+ab\},$ 
 \item $\B_4(S)=\{390b+20a+ab,\quad 402b+20a+ab\}.$ 
\end{itemize}
For instance, there are $10$ positive differences of elements in $\B_1(S)$. Hence, Lemma \ref{indispensable} requires checking if $68$ elements do not lie in $S$ for every choice of $a$ and $b$. As the positive integers not is $S$ (also known as gaps of $S$) depend on $a$ and $b$, it is difficult to foresee which gluing will have a SIFRE. On the other hand, Conditions $(2)$ and $(3)$ of Theorem \ref{gluing} hold automatically and it is sufficient to check Condition $(1)$ only. This means to check if
\begin{itemize} 
 \item  $\pm(a-b_1)\notin S_1$, for $b_1\in \{185,279,328\}$
  \item $\pm(a+b_1-b_2)\notin S_1$, for $b_1\in \{185,279,328\}$ and  $b_2\in \{390,402\}$
   \item  $\pm(b-c_1)\notin S_2$, for $c_1\in \{20\}$.
\end{itemize}
 One can use gaps of $\langle 31,37,41\rangle$ to see only $a=109$ or $a=150$ yield a situation where the first two items above hold. 

\noindent One can use gaps $\{1,2,3,6,7,11\}$ of $\langle 4,5\rangle$ to see that the last bullet holds for any 
$$b\in Q=\{19,18,17,14,13,9,21,22,23,26,27,31\}.$$ 
The values when $a=109$ with any $b\in Q$ or  $a=150$ with $b=19,17,13,23,31$ produce gluings with SIFREs. Similarly one can see using our criterion that $\langle 6,7,10\rangle$ or $\langle 8,9,11\rangle$ does not give rise to a gluing with a SIFRE.
 \end{example}

\begin{remark} \label{generalization} This section generalizes the main results of \cite{go} as we briefly explain now. Our Lemma \ref{betti} specializes to $\B_{1}(S)= \B_{1}(S_1)\cup \B_{1}(S_2)\cup \{\aa\} $, which is exactly \cite[Theorem 10]{go}. Furthermore the condition (\ref{cond1}) in our Theorem \ref{gluing}, specializes to $\mp (\aa-b_{1,j}) \notin S_j$, since $b_{0,j}=0$, for $j=1,2$. This is exactly the condition in \cite[Theorem 12]{go} by the virtue of Lemma \ref{dif} and $\alpha \in S_1 \cap S_2$. One can produce gluings with unique presentations or equivalently unique minimal generating sets which do not have SIFREs. For example, take the gluing in Example \ref{eg} with $a=355$ and  
$$b\in Q=\{19,18,17,14,13,9,21,22,23,26,27,31\}.$$ Then, $S$ has a unique presentation as the following hold:
\begin{itemize} 
 \item  $\pm(a-b_1)\notin S_1$, for $b_1\in \{185,279,328\}$
   \item  $\pm(b-c_1)\notin S_2$, for $c_1\in \{20\}$.
\end{itemize}
On the other hand, we have seen in Example \ref{eg} that the following condition does not hold:
\begin{itemize} 
  \item $\pm(a+b_1-b_2)\notin S_1$, for $b_1\in \{185,279,328\}$ and  $b_2\in \{390,402\}$.
\end{itemize} 
\end{remark}

\section{Extending Strongly Indispensable Resolutions}
We determine some semigroups having SIFREs in this section. We focus on a particular case of gluing where the second semigroup is generated by a single element. These semigroups are also known as extensions in the literature. Given an affine semigroup $S$ generated minimally by $\m_1,\dots,\m_n$, recall that an {\it extension of $S$} is an affine semigroup denoted
by $E$ and generated minimally by
$\ell \m_1,\dots,\ell \m_n$ and $\m$, where $\ell$ is a positive integer coprime to a component of 
$\m=u_1\m_1+\dots+u_n\m_n$ for some non-negative integers $u_1,\dots,u_n$. Note that $E $ is the gluing of $S_1=\ell S$ and $S_2=\N \{\m\}$, with $\aa=\ell \m$.
\begin{theorem}\label{main1}  $K[E]$ has a SIFRE if and only if $K[S]$ has a SIFRE and the condition $\pm(\m+\mathbf{b'}-\mathbf{b})\notin S$ holds, for all $\mathbf{b} \in \B_{i}(S)$, $\mathbf{b'} \in \B_{i-1}(S)$, and $1\leq i \leq \pd(S)+1$.
\end{theorem}
\begin{proof} $K[E]$ has a SIFRE if and only if $\e-\e' \notin E$ for all $\e,\e' \in \B_i(E)$ by Lemma \ref{indispensable}. It follows from Lemma \ref{betti} that $\B_i(E)=\ell \B_i(S) \cup \ell[ \B_{i-1}(S)+ \m]$, and so we have three possibilities if $\e,\e' \in \B_i(E)$:
\begin{enumerate}
\item $\e,\e' \in \ell \B_i(S),$ 
 \item $ \e,\e' \in \ell[ \B_{i-1}(S)+ \m], $
 \item $\e\in \ell \B_i(S) \:\mbox{and} \: \e' \in \ell[ \B_{i-1}(S)+ \m].$
\end{enumerate}
In the first two cases $\e-\e'=\ell (\b-\b') \notin E$ if and only if $ \b-\b' \notin S$, by Lemma \ref{dif}, which is equivalent to $K[S]$ having a SIFRE by Lemma \ref{indispensable}. In the last one, $\pm(\e-\e')=\pm\ell (\b-\b'-\m) \notin E$ if and only if $ \pm(\b-\b'-\m) \notin S$, which completes the proof.
\end{proof}

\subsection{Symmetric affine semigroups} As another application, we obtain infinitely many complete intersection semigroup rings with a SIFRE. When $E$ is symmetric (or equivalently $S$ is symmetric), it is sufficient to check  the condition above for the first half of the indices.
\begin{corollary} \label{symmetriccorollary} If $E$ is symmetric, then $K[E]$ has a SIFRE if and only if $K[S]$ has a SIFRE and $\pm(\m+\mathbf{b'}-\mathbf{b})\notin S$, for $\mathbf{b} \in \B_{i}(S)$, $\mathbf{b'} \in \B_{i-1}(S)$, and $1\leq i \leq \lfloor \pd(E)/2 \rfloor$.
\end{corollary}
\begin{proof} The proof mimics the proof of Theorem \ref{main1}, applying Lemma \ref{symmetriclemma} instead of Lemma \ref{indispensable}.
 \end{proof}
Let $n>1$ and $\{\e_i : i=1,\dots,n\}$ denote the canonical basis of $\N^n$. Let $u_1,\dots,u_n$ be some positive integers and $S$ be the semigroup generated minimally by $u_1\e_1,\dots,u_n\e_n$. It is clear that $I_S=(0)$ and thus $K[S]=K[x_1,\dots,x_n]$ has a SIFRE. 

Fix $\a=(-u_1,u_2,\dots,u_n)$, $\a_0=(0,u_2,\dots,u_n)\in S$ and consider the extensions of $S$ defined recursively as follows:
\begin{itemize}
\item $E_1=2S+\N\{\a_1\}$, where $\a_1=2\a_0-\a=(u_1,u_2,\dots,u_n)\in S$, and 
\item $E_j=2E_{j-1}+\N\{\a_j\}$, where $\a_j=\a_{j-1}+2\a_{j-2} \in E_{j-1}$, for $j\geq2$.
\end{itemize}

\begin{proposition} With the notations above, we have
\begin{enumerate}
 \item $\a_j-2\a_{j-1}=(-1)^j \a$, for all $j \geq 1$, 
 \item $\a_j +\b' -\b =u \a$, for some $u\in \Z  - \{0\}$ and for all $\b'\in \B_{i-1}(E_{j-1})$, $\b\in \B_{i}(E_{j-1})$, where $j\geq 2$ and $1\leq i \leq \lfloor j/2 \rfloor$,
 \item $K[E_j]$ has a SIFRE for all $j\geq 1$.

\end{enumerate}
\end{proposition}

\begin{proof} We use induction on $j$ in all items.
\begin{enumerate}
 \item The claim follows from the definition of $\a_1=2\a_0-\a$ when $j=1$. Assuming that the claim is true for $j=p-1$, we have $\a_p-2\a_{p-1}= -(\a_{p-1}-2\a_{p-2})=-(-1)^{p-1} \a=(-1)^{p} \a$, since $\a_p=\a_{p-1}+2\a_{p-2}$, for all $p \geq 2$.
 \item When $j=2$ and $1\leq i \leq \lfloor j/2 \rfloor=1$, we have $\a_2 +\b' -\b =\a_2-2\a_1=\a$, by Part $(1)$, for all $\b'\in \B_{0}(E_{1})=\{0\}$, $\b\in \B_{1}(E_{1})=\{2\a_1\}$, as $I_{E_1}$ is a principal ideal generated by $y^2-x_1^{a_1}\cdots x_n^{a_n}$ of $E_1$-degree $2\a_1$. Assume now that the claim is true for all indices $3\leq j \leq p-1$. We need to study $\a_p +\b' -\b$, for all $\b'\in \B_{i-1}(E_{p-1})$, $\b\in \B_{i}(E_{p-1})$, where $p\geq 4$ and $1\leq i \leq \lfloor p/2 \rfloor$. There are four cases to consider since by Lemma \ref{betti}, $\B_i(E_{p-1})=2 \B_i(E_{p-2}) \cup 2[ \B_{i-1}(E_{p-2})+ \a_{p-1}]$: \\
 Case (i): $\b'=2\cc'$, $\cc'\in \B_{i-1}(E_{p-2})$ and $\b=2\cc$, $\cc \in \B_{i}(E_{p-2})$. In this case, $\a_p +\b' -\b=\a_p -2\a_{p-1}+2(\a_{p-1}+\cc' -\cc)=(-1)^p\a+2u\a$, for some $u\in \Z  - \{0\}$, by the induction hypothesis.\\
 Case (ii): $\b'=2(\cc'+\a_{p-1})$, $\cc'\in \B_{i-2}(E_{p-2})$ and $\b=2(\cc+\a_{p-1})$, $\cc \in \B_{i-1}(E_{p-2})$. In this case, $\a_p +\b' -\b=\a_p -2\a_{p-1}+2(\a_{p-1}+\cc' -\cc)=(-1)^p\a+2u\a$, for some $u\in \Z  - \{0\}$, by the induction hypothesis.\\
 Case (iii): $\b'=2(\cc'+\a_{p-1})$, $\cc'\in \B_{i-2}(E_{p-2})$ and $\b=2\cc$, $\cc \in \B_{i}(E_{p-2})$. Taking $\d \in \B_{i-1}(E_{p-2})$ in this case, we have $\a_p +\b' -\b=\a_p -2\a_{p-1}+2(\a_{p-1}+\cc' -\d)+2(\a_{p-1}+\d -\cc)=(-1)^p\a+2u_1\a+2u_2\a$, for some $u_1,u_2\in \Z  - \{0\}$, by the induction hypothesis.\\
 Case (iv): $\b'=2\cc'$, $\cc'\in \B_{i-1}(E_{p-2})$ and $\b=2(\cc+\a_{p-1})$, $\cc \in \B_{i-1}(E_{p-2})$. So, $\a_p +\b' -\b=\a_p -2\a_{p-1}+2(\cc' -\cc)=(-1)^p\a+2(\cc' -\cc)$. Note that the proof will be complete if we show that $\cc' -\cc=v\a$, for some $v\in \Z$. Let us prove this by verifying the claim that $\cc' -\cc=v\a$, for some $v\in \Z$, and for all $\cc,\cc'\in \B_{i-1}(E_{q})$, $1\leq i \leq \lfloor p/2\rfloor$, using induction on $1\leq q \leq p-2$. For $q=1$, the claim is trivial with $v=0$ as $i=1$ and $\B_0(E_1)=\{0\}$. Assume now that it is true for $q=r-1$, and consider $\cc,\cc'\in \B_{i-1}(E_{r})$. By Lemma \ref{betti}, we have three possibilities as before and in two of them $\cc' -\cc=2(\d'-\d)$, for either $\d',\d\in \B_{i-1}(E_{r-1})$ or $\d',\d\in \B_{i-2}(E_{r-1})$. So, we are done by induction hypothesis on $q$. In the third one, $\cc' -\cc=2(a_p+\d'-\d)=2u\a$, by the induction hypothesis on $p$, where $\d\in \B_{i-1}(E_{r-1})$ or $\d'\in \B_{i-2}(E_{r-1})$.

 \item As $I_{E_1}$ is a principal ideal, the projective dimension of $K[E_1]$ is $1$ and $ \B_{0}(E_{1})=\{0\}$ and $\B_{1}(E_{1})=\{2\a_1\}$. So, when $j=1$, there is nothing to check in Corollary \ref{symmetriccorollary} as $K[S]$ has a SIFRE and $\lfloor j/2 \rfloor=0$. So, $K[E_1]$ has a SIFRE. 
Assume that the claim is true for $j=p-1$, so $K[E_{p-1}]$ has a SIFRE for all $p\geq 2$. We first note that the projective dimension of $K[E_p]$ is $p$, by Theorem \ref{res}. So, we need to verify that $\a_p +\b' -\b \notin E_{p-1}$, for all $\b'\in \B_{i-1}(E_{p-1})$, $\b\in \B_{i}(E_{p-1})$, where $p\geq 2$ and $1\leq i \leq \lfloor p/2 \rfloor$, which is true by $(2)$ as $E_{p-1}\subset \N^n$ and $u\a\notin \N^n$, for $u\in \Z  - \{0\}$. \end{enumerate}
So, $K[E_j]$ has a SIFRE for $j\geq 1$.
\end{proof}

\subsection{Numerical semigroups} In this section, we characterize extensions of some numerical semigroups having SIFRE. As the extensions of $3$-generated symmetric numerical semigroups were classified in \cite[Theorem 25]{bafrsa}, we start with $3$-generated non-symmetric numerical semigroups here. It is known that they have SIFREs (see \cite[Example 20]{bafrsa}). As a first application, we determine their extensions which have SIFREs, using the following results. 
\begin{theorem}[Herzog {\cite[Proposition 3.2]{he}}]
Let $\alpha_p$ be the smallest positive integer such that $\alpha_pm_p=\alpha_{pq}m_q+\alpha_{pr}m_r$, for some $\alpha_{pq},\alpha_{pr}\in\N$, where $\{ p,q,r\}=\{ 1,2,3\}$. Then $S=\langle m_1,m_2,m_3\rangle$ is $3$-generated not symmetric if and only if $\alpha_{pq}>0$ for all $p,q$, and $\alpha_{qp}+\alpha_{rp}=\alpha_p$, for all  $\{ p,q,r\}=\{ 1,2,3\}$. Then
$K[S]=R/(f_1,f_2,f_3)$, where $$f_1=x_1^{\alpha_1}-x_2^{\alpha_{12}}x_3^{\alpha_{13}}, f_2=x_2^{\alpha_2}-x_1^{\alpha_{21}}x_3^{\alpha_{23}}, f_3=x_3^{\alpha_3}-x_1^{\alpha_{31}}x_2^{\alpha_{32}}.$$
\end{theorem}
Although the following follows from the previous result and the classical Hilbert-Burch theorem, a detailed proof has been given by Denham in \cite[Lemma 2.5]{den}.
\begin{theorem}
If $S$ is a $3$-generated semigroup which is not symmetric then
$K[S]$ has a minimal graded free $R$-resolution
$$0\longrightarrow R^2\stackrel{\phi_2}{\longrightarrow} R^3\stackrel{\phi_1}{\longrightarrow}R {\longrightarrow}K[S]\longrightarrow 0,$$
where $\phi_1=(f_1\:f_2\:f_3)$,
and $\phi_2=\left( \begin{array}{cc}x_3^{\alpha_{23}}&x_2^{\alpha_{32}}\\x_1^{\alpha_{31}}&x_3^{\alpha_{13}}\\x_2^{\alpha_{12}}&x_1^{\alpha_{21}}\end{array}\right).$
\end{theorem}
\begin{remark} \label{3generated} As the resolution above is graded, $\B_1(S)=\{d_1,d_2,d_3\}$, where $d_p=\alpha_pm_p=\alpha_{pq}m_q+\alpha_{pr}m_r$, for all $p,q,r \in \{ 1,2,3\}$. Since the entries in $\phi_1\phi_2=0$ are $S$-homogeneous, we have $\B_2(S)=\{b_1,b_2\}$, where 
$$b_1=\aa_{23}m_3+d_1=\aa_{31}m_1+d_2=\aa_{12}m_2+d_3, \mbox{ and }$$ 
$$b_2=\aa_{32}m_2+d_1=\aa_{13}m_3+d_2=\aa_{21}m_1+d_3.$$
\end{remark}

\begin{theorem} Let $S=\langle m_1,m_2,m_3\rangle$ be a non-symmetric numerical semigroup and $E$ be an extension of $S$, where $m=u_1m_1+u_2m_2+u_3m_3$. Then, $K[E]$ has a SIFRE if and only if $0< u_p < \min\{\aa_{qp},\aa_{rp}\}$ for all $\{ p,q,r\}=\{ 1,2,3\}$. In particular, $S$ does not have an extension with a SIFRE if and only if $\aa_{pq}=1$ for some  $p,q\in\{ 1,2,3\}$.
\end{theorem}
\begin{proof} If  $K[E]$ has a SIFRE, then $m+d_i-b_j \notin S$ and $d_i-m \notin S$, for all $i,j\in\{1,2,3\}$ by Theorem \ref{main1}. By Remark \ref{3generated}, there are $i,j\in \{1,2,3\}$ such that $m+d_i-b_j=(u_p-\aa_{qp})m_p+u_qm_q+u_rm_r$.  When, $u_p\geq \aa_{qp}$ for some  $p,q\in\{ 1,2,3\}$, $m+d_i-b_j \in S$, which is a contradiction. Thus, $u_p < \aa_{qp}$ for all  $p,q\in\{ 1,2,3\}$. If $u_i=0$, for some $i$, then $d_i-m=(\aa_{ip}-u_p)m_p+(\aa_{iq}-u_q)m_q \in S$, which is a contradiction.

Conversely, assume that $0< u_p < \min\{\aa_{qp},\aa_{rp}\}$ for all $\{ p,q,r\}=\{ 1,2,3\}$. We claim $m-d_p=u_pm_p+(u_q-\aa_{pq})m_q+(u_r-\aa_{pr})m_r \notin S$. If not, $m-d_p=v_pm_p+v_qm_q+v_rm_r$, for some $v_i \in \N$, and so $(u_p-v_{p})m_p=(u_q+v_q-\aa_{pq})m_q+(u_r+v_r-\aa_{pr})m_r>0$ which contradicts to $\aa_p$ being the smallest positive integer with this property. Next, we prove $d_p-m=(\aa_{p}-u_p)m_p-u_qm_q-u_rm_r \notin S$. If not, $d_p-m=v_pm_p+v_qm_q+v_rm_r$, for some $v_i \in \N$, and so $(\aa_{p}-u_p-v_p)m_p=(u_q+v_q)m_q+(u_r+v_r)m_r>0$, which contradicts the fact that $\aa_p$ is the smallest positive integer with this property.

By \cite[Corollary 11]{bafrsa}, $PF(S)=\{b_1-N,b_2-N\}$, where $N=m_1+m_2+m_3$, are the pseudo-Frobenius elements of $S$. In particular, $b_i-N\notin S$. This implies that $b_i-m-d_j \notin S$, since $b_i-N=(b_i-m-d_j)+(\aa_p+u_p-1)m_p+(u_q-1)m_q+(u_r-1)m_r$. Finally, by Remark \ref{3generated}, for all $i,j$ there are $p,q,r$ such that $m+d_i-b_j=(u_p+\aa_{ip})m_p+(u_q-\aa_{pq})m_q+(u_r-\aa_{pr})m_r$. If $m+d_i-b_j=v_pm_p+v_qm_q+v_rm_r$, for some $v_i \in \N$, then $(u_p+\aa_{ip}-v_p)m_p=(v_q+\aa_{pq}-u_q)m_q+(v_r+\aa_{pr}-u_r)m_r >0$, which implies that $u_p+\aa_{ip}-v_p\geq \aa_p$, as $\aa_p$ is the smallest with this property. But this contradicts the assumption $u_p < \min\{\aa_{qp},\aa_{rp}\}$.
\end{proof}
\begin{example} Take $S=\langle 7,9,10 \rangle$. Then
$K[S]=R/(f_1,f_2,f_3)$, where $$f_1=x_1^{4}-x_2^{2}x_3^{}, \quad f_2=x_2^{3}-x_1^{}x_3^{2}, \quad f_3=x_3^{3}-x_1^{3}x_2^{}.$$
Since $\aa_{13}=1$, no extension of $S$ will have a SIFRE. On the other hand, the following semigroups will lead
 to infinitely many families of extensions having SIFREs. For $S=\langle 31,37,41 \rangle$, Macaulay2 computes 
 the following generators $$f_1=x_1^{9}-x_2^{2}x_3^{5}, \quad f_2=x_2^{5}-x_1^{2}x_3^{3}, 
 \quad f_3=x_3^{8}-x_1^{7}x_2^{3}.$$ So, $u_1=1,u_2=1$ and $1\leq u_3\leq 2$ give $m=109$ and $m=150$, respectively. 
 Hence, $E=\langle 31\ell,37\ell,41\ell,109 \rangle$ and $E=\langle 31\ell,37\ell,41\ell,109 \rangle$ have  SIFREs,
  for any $\ell$, with $\gcd(\ell,109)=1$ and with $\gcd(\ell,150)=1$, respectively. Similarly, $S=\langle 67,91,93 
  \rangle$ leads to $6$ and $S=\langle 71,93,121 \rangle$ leads to $14$ different infinite families having SIFREs.
\end{example}

Now, we study extensions of a symmetric $4$-generated not complete intersection numerical semigroup using the following theorem.
\begin{theorem}[Bresinsky {\cite[Theorem 5, Theorem 3] {br}\label{br}}]
The semigroup $S$ is $4$-generated symmetric, non-complete intersection if and only if there are integers $\alpha_i$ and
$\alpha_{ij}$, such that
$0<\alpha_{ij}<\alpha_i$, for all $i,j$, with
$ \alpha_1=\alpha_{21}+\alpha_{31}, \alpha_2=\alpha_{32}+\alpha_{42},
\alpha_3=\alpha_{13}+\alpha_{43}, \alpha_4=\alpha_{14}+\alpha_{24}\;\mbox{ and}$
\begin{eqnarray*} m_1=\alpha_2\alpha_3\alpha_{14}+\alpha_{32}\alpha_{13}\alpha_{24},&\quad
m_2=\alpha_3\alpha_4\alpha_{21}+\alpha_{31}\alpha_{43}\alpha_{24}, \\
m_3=\alpha_1\alpha_4\alpha_{32}+\alpha_{14}\alpha_{42}\alpha_{31},&\quad
m_4=\alpha_1\alpha_2\alpha_{43}+\alpha_{42}\alpha_{21}\alpha_{13}. 
\end{eqnarray*}
Then, $K[S]=R/(f_1,f_2,f_3,f_4,f_5)$, where
\begin{eqnarray*}f_1=x_1^{\alpha_1}-x_3^{\alpha_{13}}x_4^{\alpha_{14}},& 
f_2=x_2^{\alpha_2}-x_1^{\alpha_{21}} x_4^{\alpha_{24}},& 
f_3=x_3^{\alpha_3}-x_1^{\alpha_{31}}x_2^{\alpha_{32}},\\
f_4=x_4^{\alpha_4}-x_2^{\alpha_{42}}x_3^{\alpha_{43}},& \quad
f_5=x_3^{\alpha_{43}}x_1^{\alpha_{21}}-x_2^{\alpha_{32}}x_4^{\alpha_{14}}.
\end{eqnarray*}
\end{theorem}

 If $S=\langle m_1,m_2,m_3,m_4\rangle$ is a symmetric non-complete intersection numerical semigroup then $K[S]$ has a SIFRE by  \cite[Theorem 27]{bafrsa}. Let $E$ be an extension of $S$ with $m=u_1m_1+u_2m_2+u_3m_3+u_4m_4$. Then, we have the following result.
\begin{theorem} $K[E]$ has a SIFRE if and only if $u_p < \min\{\aa_{qp},\aa_{rp}\}$ for all $ p,q,r \in \{ 1,2,3,4\}$ and at most one $u_p=0$ such that
\begin{eqnarray*}
&u_1=0 &\implies  \aa_{32}-\aa_{42}< u_2 \quad \text{or} \quad \aa_{13}-\aa_{43}< u_3, \\
&u_2=0& \implies  \aa_{43}-\aa_{13}< u_3 \quad \text{or} \quad \aa_{24}-\aa_{14}< u_4, \\
&u_3=0&\implies  \aa_{31}-\aa_{21}< u_1 \quad \text{or} \quad \aa_{14}-\aa_{24}< u_4, \\
&u_4=0&\implies  \aa_{21}-\aa_{31}< u_1 \quad \text{or} \quad \aa_{42}-\aa_{32}< u_2.
\end{eqnarray*}
\end{theorem}
\begin{proof} We use \cite[Corollary 13]{bafrsa} and Theorem \ref{br} for all the relations involving $a_i$ and $d_i$, where $d_i=\deg_S(f_i)$, and $a_i$ is the $S$-degree of a first syzygy, for $i=1,\dots,5$. We first prove the necessity of these conditions. Assume $K[E]$ has a SIFRE. Then by Corollary \ref{symmetriccorollary}, $m+d_j-a_k\notin S$ and $d_i-m\notin S$, Given $\aa_{pi}$ there are $j,k$ with $d_j-a_k=-\aa_{pi}m_i$. So, if $u_i\geq\aa_{pi}$, then $\displaystyle m+d_j-a_k=\sum_{q \neq i}u_qm_q+(u_i-\aa_{pi})m_i \in S$. Therefore, $u_p < \min\{\aa_{qp},\aa_{rp}\}$ for all $\{ p,q,r\}=\{ 1,2,3,4\}$. If $u_p=u_q=0$, then $d_i-m=(\aa_{jr}-u_r)m_r+(\aa_{ts}-u_s)m_s-u_pm_p-u_qm_q\in S$. So, at most one $u_p=0$. If  $\aa_{32}-\aa_{42}\geq u_2$ and $\aa_{13}-\aa_{43}\geq u_3$ when $u_1=0$, then $a_4-d_4-m=(\aa_{32}-\aa_{42}- u_2)m_2+(\aa_{13}-\aa_{43}- u_3)m_3+(\aa_{14}-u_4)m_4\in S$. The others are shown similarly. Next, we prove sufficiency. Assume $u_p < \min\{\aa_{qp},\aa_{rp}\}$ for all $\{ p,q,r\}=\{ 1,2,3,4\}$ and at most one $u_p=0$. Then, $d_i-m=\sum v_jm_j$ implies $d_i=\sum (u_j+v_j)m_j$. Since $f_i$ is indispensable, there are only two monomials with $S$-degree $d_i$. So, $\sum (u_j+v_j)m_j$ must be $\aa_im_i$ or $\aa_{pq}m_q+\aa_{rs}m_s$. In any case, at least two $u_j=0$, which is a contradiction. So, $m-d_i\notin S$. If $m-d_i=\sum v_jm_j$, then $d:=(u_i-v_i)m_i+(u_j-v_j)m_j=(u_q-\aa_{pq}-v_q)m_q+(u_s-\aa_{rs}-v_s)m_s>0$. Since $u_i-v_i<\aa_i$ and $u_j-v_j<\aa_j$, we get $u_i>v_i$ and $u_j>v_j$ by the minimality of $\aa_i$ and $\aa_j$. But then $d<_S d_s=\aa_{pi}m_i+\aa_{qj}m_j$, which contradicts the minimality of $d_s$. So, $m-d_i\notin S$. For, $i\neq j$ and $(i,j)\notin \{(1,3),(2,4)\}$, we have $a_i-d_j=\aa_{pq}m_q$. So, if $a_i-d_j -m=\sum v_jm_j$, then $(\aa_{pq}-u_q-v_q)m_q=\sum_{j\neq q} (u_j+v_j)m_j \geq 0$.  Since $\aa_{pq}-u_q-v_q<\aa_q$, all $u_i=0$ for $i\neq q$. So, $a_i-d_j -m\notin S$. If $a_1-d_3-m=\sum v_jm_j$, then $(\aa_{2}-u_2-v_2)m_2=(\aa_{13}+u_3+v_3)m_3+ (u_1+v_1)m_1+(u_4+v_4)m_4 > 0$. By the minimality of $\aa_{2}$, $u_2=v_2=0$ in which case we have a third monomial $x_1^{u_1+v_1}x_3^{\aa_{13}+u_3+v_3}x_4^{u_4+v_4}$ of $S$-degree $d_2$, which is a contradiction to the indispensability of $f_2$. Similarly, $a_2-d_4-m=\aa_1m_1-\aa_{42}m_2-m\notin S$. To prove that $a_i-d_i-m \notin S$, it suffices to see $m+2d_i-N \in S$, as the Frobenius number of $S$, which is the biggest integer not in $S$, is $a_i+d_i-N=(a_i-d_i-m)+(m+2d_i-N)$ by \cite[Corollary 14]{bafrsa}, where $N=\sum_j m_j$. If all $u_i>0$ then $m-N\in S$, so $m+2d_i-N \in S$. If $u_p=0$ for some $p\in\{1,2,3,4\}$, then $m+2d_i-N=(d_i+m-m_j-m_q-m_r)+(d_i-m_p) \in S$, except for $(i,p)\in\{(4,1),(1,2),(2,3),(3,4)\}$. When $u_1=0$, we have $a_4-d_4-m=(\aa_{32}-\aa_{42}-u_2)m_2+(\aa_{13}-\aa_{43}-u_3)m_3+(\aa_{14}-u_4)m_4$. If $a_4-d_4-m=\sum v_jm_j$, then either $\aa_{32}-\aa_{42}\geq u_2$ or $\aa_{13}-\aa_{43} \geq u_3$ as $0<\aa_{14}-u_4< \aa_4$. If $\aa_{32}-\aa_{42}\geq u_2$ then $\aa_{13}-\aa_{43} < u_3$ and thus we have
$$(\aa_{32}-\aa_{42}-u_2-v_2)m_2+(\aa_{14}-u_4-v_4)m_4=v_1m_1+(\aa_{43}-\aa_{13}+u_3+v_3)m_3.$$
This gives a binomial in $I_S$ of degree $d$ with $d<_S d_5$, which is a contradiction to the minimality of $d_5$. If $\aa_{13}-\aa_{43} \geq u_3$, then we get similarly a binomial of $S$-degree less than $d_1$. The other cases can be done the same way. Finally, we prove that $m+d_j-a_i \notin S$. For $i\neq j$ and $(i,j)\notin \{(1,3),(2,4)\}$, we have $m+d_j-a_i=\sum_{p\neq q}u_pm_p+(u_q-\aa_{rq})m_q$. If $m+d_j-a_i=\sum v_pm_p \in S$, then $(u_r-v_r)m_r+(u_s-v_s)m_s=(v_p-u_p)m_p+(\aa_{rq}-u_q+v_q)m_q$, as the other cases contradicts the minimality of some $\aa_t$. But this gives a binomial in $I_S$ of $S$-degree less than $d_t=\aa_{pr}m_r+\aa_{qs}m_s$, which is a contradiction. Now, $d_1-a_1=\aa_{31}m_1-\aa_{43}m_3-\aa_{24}m_4$. If $m+d_1-a_1=\sum v_jm_j$, then $(u_2-v_2)m_2+(\aa_{31}+u_1-v_1)m_1=(\aa_{43}-u_3+v_3)m_3+(\aa_{24}-u_4+v_4)m_4>0$. If one term of the left hand side is negative then we get a contradiction to the minimality of $\aa_1$ and $\aa_2$. So, this gives a binomial in $I_S$ of $S$-degree $d$. Only $d_3$ may be less than $d$, so $\aa_{32}\leq u_2-v_2\leq u_2$, which is a contradiction. The rest is similar and we are done.
\end{proof}
\begin{remark} Using the formulas in Theorem \ref{br}, one can now produce infinitely many symmetric non complete intersections $S=\langle m_1,m_2,m_3,m_4\rangle$ having SIFREs.
\end{remark}

There is a classification of  $4$-generated pseudo symmetric semigroups having SIFRE in \c{S}ahin and \c{S}ahin \cite{pseudo}. The next result reveals that none of the extensions of these semigroups have a SIFRE. 
\begin{theorem} Let $S$ be a $4$-generated pseudo symmetric semigroup and $E$ be one of its extensions. Then $K[E]$ does not have a SIFRE.
\end{theorem}
\begin{proof} By the proof of Theorem~2.5 in \cite{pseudo}, we have $\B_1(S)=\{d_1,\dots,d_6\}$, $\B_2(S)=\{b_1,\dots,b_6\}$ and $\B_3(S)=\{c_1,c_2\}$, where $$c_1=b_1+m_4=b_2+m_1=b_4+m_3=b_5+m_2.$$
Thus, if $m=u_1m_1+\cdots+u_4m_4$ with $u_j>0$, then there is some $i$ such that $m+b_i-c_1=m-m_j \in S$. The result now follows from Theorem \ref{main1}.
\end{proof}

\section*{Acknowledgements} The authors would like to thank Anargyros Katsabekis for his
suggestions on the preliminary version of the paper. They also thank the referee for helpful comments and suggestions. All the examples were computed by using the computer algebra system Macaulay $2$, see \cite{Mac2}.

\bibliographystyle{amsplain}

\end{document}